\documentclass[12pt]{article}
\usepackage{amsfonts}

\usepackage[all]{xy}
\usepackage{amsmath}

\def\squarebox#1{\hbox to #1{\hfill\vbox to #1{\vfill}}}
\newcommand{\qed}{\hspace*{\fill}
\vbox{\hrule\hbox{\vrule\squarebox{.667em}\vrule}\hrule}\smallskip}
\newtheorem{teorema}{Theorem}[section]
\newtheorem{lema}[teorema]{Lemma}
\newtheorem{corolario}[teorema]{Corollary}
\newtheorem{proposicao}[teorema]{Proposition}

\newenvironment{proof}{\noindent {\bf Proof:}}{\hfill $\qed $ \newline}

\begin{document}

\title{The Topological Entropy of \\ Endomorphisms of Lie Groups}
\author{Mauro Patr\~{a}o\footnote{
    Department of Mathematics, University of Bras\'ilia, Brazil. 
    mpatrao@mat.unb.br
  }}
\maketitle

\begin{abstract}
In this paper, we determine the topological entropy $h(\phi)$ of a continuous endomorphism $\phi$ of a Lie group $G$. This computation is a classical topic in ergodic theory which seemed to have long been solved. But, when $G$ is noncompact, the well known Bowen's formula for the entropy $h_{d}(\phi)$ associated to a left invariant distance $d$ just provides an upper bound to $h(\phi)$, which is characterized by the so called variational principle. We prove that
\[
h\left(\phi\right) = h\left(\phi|_{T(G_\phi)}\right)
\]
where $G_\phi$ is the maximal connected subgroup of $G$ such that $\phi(G_\phi) = G_\phi$, and $T(G_\phi)$ is the maximal torus in the center of $G_\phi$. This result shows that the computation of the topological entropy of a continuous endomorphism of a Lie group reduces to the classical formula for the topological entropy of a continuous endomorphism of a torus. Our approach explores the relation between null topological entropy and the nonexistence of Li-Yorke pairs and also relies strongly on the structure theory of Lie groups.
\end{abstract}

\noindent \textit{AMS 2010 subject classification}: Primary: 37B40, 22D40; Secondary: 37A35, 22E99.

\noindent \textit{Key words:} Topological entropy, endomorphisms of Lie groups, Li-Yorke pairs.

\section{Introduction}

The computation of the topological entropy of a continuous endomorphism $\phi$ of a Lie group $G$ is a classical topic in ergodic theory which seemed to have long been solved. In \cite{bowen-endomorphism}, Bowen proved that, when $G$ is connected, the entropy of $\phi$ associated to a left invariant distance $d$ is given by
\[
h_d(\phi) = \sum_{|\lambda| > 1} \log\left|\lambda\right|
\]
where $\lambda$ are eigenvalues of the derivative $d\phi_1$, generalizing a result which was already known when $G$ is a torus. When $T: X \to X$ is a continuous map and $X$ is a compact metric space, Bowen, Dinaburg, and Goodman (see \cite{bowen-endomorphism,dinaburg,goodman}) proved the so-called variational principle 
\[
\sup_{\mu} h_{\mu}(T) = h(T) = h_d(T)
\]
where $h_\mu(T)$ is the entropy of $T$ associated to the probability measure $\mu$, introduced by Kolmogorov and Sinai in \cite{sinai}, $h(T)$ is the topological entropy of $T$, introduced by Adler, Konheim and MacAndrew  in \cite{akm}, while $h_d(T)$ is the entropy of $T$ associated to the distance $d$, introduced by Bowen in \cite{bowen-endomorphism} and by Dinaburg in \cite{dinaburg}. Hence, when $G$ is compact and connected, Bowen's formula in fact provides the topological entropy of an endomorphism of $G$. For example, consider the continuous endomorphism $\phi$ of the two dimensional torus, given by $\phi(z,w) = (z^2,w^2)$, where $z,w \in S^1 \subset \mathbb{C}^*$. Since $d\phi_1 = \left(\begin{smallmatrix}2&0\\0&2\end{smallmatrix}\right)$, Bowen's formula implies that $h(\phi) = h_d(\phi) = 2\log(2)$. 

But when $G$ is not compact, Bowen's formula just provides an upper bound to the topological entropy. When $X$ is a separable locally compact metric space, it was proved, first in \cite{patrao:entropia}, assuming that $T$ is proper, and then in \cite{caldas-patrao:entropia} for every continuous map $T$, that
\[
\sup_{\mu} h_{\mu}(T) = h(T) = \min_d h_d(T)
\]
where $h(T)$ is a natural generalization of the topological entropy for locally compact metric spaces. Using this generalization of the variational principle and the Poincar\'e Lemma, it follows that $h(T) = h\left(T|_{\overline{\mathcal{R}_T}}\right)$, where $\mathcal{R}_T$ is the recurrent set of $T$. For example, consider the endomorphism $\phi = \left(\begin{smallmatrix}2&0\\0&2\end{smallmatrix}\right)$ of $\mathbb{R}^2$. It is easy to see that $\mathcal{R}_\phi = 0$, which implies that $h(\phi) = h\left(\phi|_0\right) = 0$. On the other hand, since $d\phi_1 = \left(\begin{smallmatrix}2&0\\0&2\end{smallmatrix}\right)$, Bowen's formula implies that $h_d(\phi) = 2\log(2) > 0$, which is just an upper bound to the topological entropy. As another example, consider the continuous endomorphism $\phi$ of $\mathbb{C}^*$, given by $\phi(z) = z^2$. It is easy to see that $\mathcal{R}_\phi \subset S^1$, which implies again that $h\left(\phi|_{S^1}\right) = \log(2)$. On the other hand, since $d\phi_1 = \left(\begin{smallmatrix}2&0\\0&2\end{smallmatrix}\right)$, Bowen's formula implies that $h_d(\phi) = 2\log(2) > \log(2)$, which is again just an upper bound to the topological entropy. Since the corresponding maximal tori in the respective centers are given by $T(\mathbb{R}^2) = 0$ and $T(\mathbb{C}^*) = S^1$, these examples led us to conjecture that $h\left(\phi\right) = h\left(\phi|_{T(G)}\right)$ for a continuous surjective endomorphism $\phi$ of a connected Lie group $G$, where $T(G)$ is the maximal torus in the center of $G$. In \cite{caldas-patrao:endomorfismos}, this conjecture was proved when $G$ is nilpotent and when $G$ is reductive. In \cite{caldas-patrao:entropia}, the surjectivity hypothesis was removed and it was proved that, for every continuous endomorphism $\phi$ of a connected nilpotent or reductive Lie group $G$, $h\left(\phi\right) = h\left(\phi|_{T(G_\phi)}\right)$, where $G_\phi$ is the maximal subgroup of $G$ such that $\phi(G_\phi) = G_\phi$.

The main result of the present paper is that
\[
h\left(\phi\right) = h\left(\phi|_{T(G_\phi)}\right)
\]
where $G_\phi$ is the maximal connected subgroup of $G$ such that $\phi(G_\phi) = G_\phi$, and $T(G_\phi)$ is the maximal torus in the center of $G_\phi$ (see Theorem \ref{theo-entropia-endomorphism}). This result solves the above conjecture in full generality and shows that the computation of the topological entropy of a continuous endomorphism of a Lie group reduces to the classical formula for the topological entropy of a continuous endomorphism of a torus. The example of minimal dimension where the previous results do not apply is the group of the orientation preserving isometries of the plane, which is a solvable group isomorphic to semidirect product of $\mathbb{R}^2$ by $S^1$. Since the center of this group is trivial, the entropy of every continuous and sujective endomorphism vanishes. 

The main difficulties in order to prove this result are the following. First it remains an open question if it is possible to fully generalize to locally compact spaces a result proved by Bowen in \cite{bowen-endomorphism} about the relation of topological entropies of maps in a compact principal bundle. However we are able to prove a partial generalization of this result, assuming that the maps are proper and that the structural group is compact, and obtaining only one of the inequalities (see Proposition \ref{prop-bowen-generalizado}). Hence we have to explore the relation between null topological entropy and the nonexistence of Li-Yorke pairs. It turned out that, when we are considering fiber bundles given by quotient of Lie groups, the existence of Li-Yorke pairs is much easier to deal with than the positivity of the topological entropy (see Proposition \ref{prop-li-yorke-G-H}). The second main difficulty is to show that, if $T(G)$ is trivial, then $\phi$ has no Li-Yorke pair (see Corollary \ref{cor-tg-trivial-no-li-yorke-pair}). The problem here is the following. When $G$ is nilpotent, if $T(G)$ is trivial, then $G$ is simply-connected, which implies that $\phi$ is conjugated by the exponential map to $d\phi_1$. But when $G$ is solvable, $\phi$ might not be conjugated to a linear map, since $G$ might not be simply-connected even when $T(G)$ is trivial. And the reason is that the center of a simply-connected solvable Lie group can be disconnected, which never happens with a simply-connected nilpotent Lie group. The proof of the main technical result here (see Theorem \ref{teo-toral-finite-order}), which shows that the induced endomorphism on the maximal torus of the quotient of a connected solvable Lie group by its simply-connected nilpotent radical has finite order, relies heavily on the so called Fundamental Theorem on the Center for connected Lie groups (see Theorem 14.2.8 of \cite{hilgert-neeb}).

The paper is organized as follows. In Section \ref{preliminaries}, we collect the main concepts and results of Lie theory which are necessary in the subsequent sections. In Section \ref{maximal}, we prove some necessary results related to the maximal torus $T(G)$ of the center of a Lie group $G$. In Section \ref{finite}, we prove the main technical result of the paper. In Section \ref{li-yorke}, we obtain some useful results about the existence and nonexistence of Li-Yorke pairs. Finally, in Section \ref{topological}, we prove the main results of this paper, about the computation of the topological entropy of endomorphisms of Lie groups.

\section{Preliminaries on Lie theory}\label{preliminaries} 

In this section, we collect the main concepts and results that are used in the next sections. Given a Lie group $G$ with Lie algebra $\mathfrak{g}$, the connected component of the identity of $G$ is denoted by $G_0$. The center of $G$
\begin{equation}
Z(G) = \{h \in G : gh=hg, \mbox{ for all } g \in G\}
\end{equation}
is a closed normal subgroup of $G$ whose Lie algebra is the center of $\mathfrak{g}$
\begin{equation}
\mathfrak{z}(\mathfrak{g}) = \{H \in \mathfrak{g} : [H,X] = 0, \mbox{ for all } X \in \mathfrak{g}\}
\end{equation}
which is an ideal of $\mathfrak{g}$ (see Lemma 11.1.1 of \cite{hilgert-neeb}). An isomorphism between Lie groups is an isomorphism of groups which is also a diffeomorphism and, when there exists an isomorphism between the Lie groups $G$ and $H$, we denote this by $G \simeq H$. The proof following result can be found in Section 9.5 of \cite{hilgert-neeb}, specially in Theorem 9.5.4.

\begin{proposicao}
Let $G$ be a connected Lie group with Lie algebra $\mathfrak{g}$. Then there exist a simply connected Lie group $\widetilde{G}$ and a discrete subgroup $\widetilde{\Gamma}$ of its center such that $G \simeq \widetilde{G}/\widetilde{\Gamma}$. The Lie group $\widetilde{G}$ is unique up to isomorphism and has Lie algebra $\mathfrak{g}$. Each endomorphism $\phi$ of $G$ is induced by a unique endomorphism $\widetilde{\phi}$ of $\widetilde{G}$ such that $\widetilde{\phi}\left(\widetilde{\Gamma}\right) \subset \widetilde{\Gamma}$.
\end{proposicao}

A consequence of this theorem is that every connected abelian Lie group is isomorphic to the product of a vector space by a torus, which is a connected and compact abelian Lie group (see Example 9.5.6 of \cite{hilgert-neeb}). In particular, the connected component $Z(G)$ of the center of $G$ is the product of a vector space by a torus denoted by $T(G)$. It also implies the following result.

\begin{proposicao}\label{torus-discrete}
If $T$ is a torus, then $T \simeq \mathbb{R}^n/\mathbb{Z}^n$ and the group of the automorphisms of $T$ is isomorphic to the discrete Lie group $\mathrm{GL}(n,\mathbb{Z})$. 
\end{proposicao}

In every Lie algebra $\mathfrak{g}$ of a Lie group $G$, there exist a unique maximal solvable ideal $\mathfrak{r}$ (see Proposition 5.4.3 of \cite{hilgert-neeb}), called the solvable radical of $\mathfrak{g}$, and a unique maximal nilpotent ideal $\mathfrak{n}$ (see Exercise 5.2.5 of \cite{hilgert-neeb}), called the nilpotent radical of $\mathfrak{g}$. The connected subgroups $R$ and $N$ of $G$ generated by, respectively, $\mathfrak{r}$ and $\mathfrak{n}$ are called the solvable and nilpotent radicals of $G$. Since the derived subalgebra $\mathfrak{r}'$ of $\mathfrak{r}$ is a nilpotent ideal of $\mathfrak{g}$, it follows that $\mathfrak{r}' \subset \mathfrak{n} \subset \mathfrak{r}$ and hence that $R' \subset N \subset R$, where $R'$ is the derived subgroup of $R$. Since $\widetilde{R}$ and $\widetilde{N}$ are simply-connected, we also have the following result (see Proposition 11.2.15 of \cite{hilgert-neeb}).

\begin{proposicao}\label{vector-space}
The quotient $\widetilde{R}/\widetilde{N}$ is an abelian simply-connected Lie group and hence is isomorphic to a vector space.
\end{proposicao}

The solvable radicals can also be characterized by the following result (see Lemma 5.6.1 of \cite{hilgert-neeb}).

\begin{proposicao}\label{radicals}
The solvable radical $\mathfrak{r}$ is the unique ideal of $\mathfrak{g}$ such that $\mathfrak{g}/\mathfrak{r}$ is semi-simple. The solvable radical $R$ is the unique connected normal subgroup of $G$ such that $G/R$ is semi-simple. 
\end{proposicao}

The adjoint representation of $G$ is the map given by $\mathrm{Ad}(g) = d(C_g)_1$, where $C_g(h) = ghg^{-1}$ is the conjugation by $g \in G$. The adjoint representation of $\mathfrak{g}$ is the map given by $\mathrm{ad}(X)Y = [X,Y]$, where $X,Y \in \mathfrak{g}$. When $G$ is a linear Lie group, we have that $\mathrm{Ad}(g)X = gXg^{-1}$ and that $\mathrm{ad}(X)Y = XY-YX$. In the general linear algebra $\mathfrak{gl}(n,\mathbb{C})$ of all complex square matrices, we can consider the Cartan subalgebra $\mathfrak{h}$ of the diagonal matrices of $\mathfrak{g}$. It is easy to verify that
\begin{equation}
\mathrm{ad}(H)E_{ij} = \alpha_{ij}(H)E_{ij}
\end{equation}
where $E_{ij}$ is the square matrix which has one in the position $ij$ and zeros in the other positions and $\alpha_{ij}(H) = h_{i} - h_{j}$, where $H = \mathrm{diag}(h_1,\ldots,h_n)$. The linear map $\alpha_{ij}$ is called a root associated to the Cartan subalgebra $\mathfrak{h}$, the subspace $\mathfrak{g}_{\alpha_{ij}}$ generated by $E_{ij}$ is called the root space associated to $\alpha_{ij}$, and 
\begin{equation}
\mathfrak{gl}(n,\mathbb{C})
= 
\mathfrak{h} \oplus \sum_{ij} \mathfrak{g}_{\alpha_{ij}}
\end{equation}
is called the root space decomposition of $\mathfrak{gl}(n,\mathbb{C})$ associated to $\mathfrak{h}$. A subalgebra $\mathfrak{k}$ of $\mathfrak{g}$ is compactly embedded if the subgroup generated by $\mathrm{ad}(\mathfrak{k})$ has compact closure in $\mathfrak{gl}(\mathfrak{g})$ and it is maximal if it is not properly contained in any other. The following result is known as the Fundamental Theorem on the Center (see Exercise 14.2.1, Lemma 14.2.4 and Theorems 14.2.7 and 14.2.8 of \cite{hilgert-neeb}).

\begin{teorema}\label{fundamental-center}
The image of a maximal compactly embedded abelian subalgebra by an automorphism is a maximal compactly embedded abelian subalgebra. Two maximal compactly embedded abelian subalgebras are conjugated by an inner automorphism. The center of a connected Lie group $G$ is given by 
\begin{equation}
Z\left(G\right) = \exp\{X \in \mathfrak{h} : \mbox{the spectrum of } \mathrm{ad}(X) \subset 2\pi i \mathbb{Z} \}
\end{equation}
where $\mathfrak{h}$ is any maximal compactly embedded abelian subalgebra of $\mathfrak{g}$. If $\exp(X) \in Z\left(G\right)$, then $\mathrm{ad}(X)$ is semi-simple and its spectrum is in $2\pi i \mathbb{Z}$.
\end{teorema}

\section{Maximal torus}\label{maximal}

In this section, we prove some results related to the maximal torus $T(G)$ of the center of a Lie group $G$. We start with a well known lemma whose proof is presented for the sake of the reader's convenience.

\begin{lema}\label{lema-torus-normal-central}
Let $G$ be a connected Lie group and $T$ be a torus which is a normal subgroup of $G$. Then $T$ is a subgroup of $Z(G)$. 
\end{lema}
\begin{proof}
Let $C_g: G \to G$ be the conjugation by $g \in G$. Since $T$ is a normal subgroup of $G$, we can consider the restriction $C_g|_T: T \to T$. Hence $\Phi$, given by $\Phi(g) = C_g|_T$, is a continuous homomorphism from the connected group $G$ to the discrete group of the automorphisms of $T$ (see Proposition \ref{torus-discrete}). It follows that $\Phi$ is the trivial homomorphism, showing that $C_g(h) = h$, for all $g \in G$ and for all $h \in T$. This shows that $T \subset Z(G)$.
\end{proof}

The next result relates the maximal torus of the center of a given Lie group to the maximal torus of the center of its solvable and nilpotent radicals.

\begin{proposicao}\label{prop-toral-components-coincide}
Let $G$ be a connected Lie group, $R$ be its solvable radical and $N$ be its nilpotent radical. Then their toral components coincide
\begin{equation}
T(G) = T(R) = T(N)
\end{equation}
\end{proposicao}
\begin{proof}
Since $T(G)$ is a normal connected abelian subgroup of $G$, we have that $T(G) \subset R$ and hence $T(G) \subset T(R)$. On the other hand, let $C_g: G \to G$ be the conjugation by $g \in G$. Since $R$ is a characteristic subgroup of $G$, we can consider the restriction $C_g|_R: R \to R$. Since $T(R)$ is a characteristic subgroup of $R$, we can consider the restriction $C_g|_{T(R)}: T(R) \to T(R)$. This shows that $T(R)$ is a normal subgroup of $G$ which is a torus. By Lemma \ref{lema-torus-normal-central}, it follows that $T(R) \subset Z(G)$, and thus $T(R) \subset T(G)$. The proof that $T(G) = T(N)$ is similar to the previous one, just replacing $R$ by $N$. 
\end{proof}

We conclude this section with the following result, which is necessary for the connection between the nonexistence of Li-Yorke pairs and the main result of the next section.

\begin{proposicao}\label{prop-toral-components-trivial}
Let $G$ be a connected Lie group and $R$ be its solvable radical. We have that $R/T(R)$ is the solvable radical of $G/T(G)$ and 
\begin{equation}
T(G/T(G)) = T(R/T(R))
\end{equation}
is the trivial group. Furthermore, the nilpotent radical of $G/T(G)$ is simply-connected.
\end{proposicao}
\begin{proof}
Since $T(G) = T(R)$, by Proposition \ref{prop-toral-components-coincide}, we have that $R/T(R)$ is a connected and solvable normal subgroup of $G/T(G)$ and that
\begin{equation}
\frac{G/T(G)}{R/T(R)} \simeq \frac{G}{R}
\end{equation}
is semi-simple. By Proposition \ref{radicals}, it follows that $R/T(R)$ is the solvable radical of $G/T(G)$. Consider the canonical homomorphism $\pi: R \to R/T(R)$. We have that $\pi^{-1}\left(T(R/T(R))\right)$ is a compact, connected, and solvable Lie group, and thus it is a torus. Since $T(R/T(R))$ is a normal subgroup, it follows that the subgroup $\pi^{-1}\left(T(R/T(R))\right)$ is also normal. By Lemma \ref{lema-torus-normal-central}, it follows that $\pi^{-1}\left(T(R/T(R))\right) \subset Z(R)$, and thus that $\pi^{-1}\left(T(R/T(R))\right) \subset T(R)$, which implies that $T(R/T(R))$ is trivial. It follows from Proposition \ref{prop-toral-components-coincide} that $T(G/T(G)) = T(R/T(R))$. It also follows that the toral component of the nilpotent radical of $G/T(G)$ is also trivial, which implies that it is simply-connected, since it is nilpotent (see Proposition 8 of \cite{caldas-patrao:endomorfismos}).
\end{proof}

\section{Finite order}\label{finite}

In this section, we prove the main technical result of the paper, which shows that the induced endomorphism on the maximal torus of the quotient of a connected solvable Lie group by its simply-connected nilpotent radical has finite order. We first need some preliminary results.

\begin{proposicao}\label{prop-diagram}
Let $R$ be a connected solvable group, $\widetilde{R}$ be its universal covering, and $\widetilde{\Gamma}$ be the discrete subgroup of $Z\left(\widetilde{R}\right)$ such that $R = \widetilde{R}/\widetilde{\Gamma}$. Denote by $N$ and $\widetilde{N}$ the the nilpotent radicals of, respectively, $R$ and $\widetilde{R}$, such that $N = \widetilde{\pi}(\widetilde{N})$. Then there exist isomorphisms $\psi$ and $\psi_2$ such that the following diagram commutes
\begin{displaymath}
    \xymatrix{
	R \ar[d]_{\pi} & \ar[l]_{\widetilde{\pi}} \widetilde{R} \ar[d]^{\pi_1} \ar[r]^{\widetilde{\pi}_2} & V \ar[d]^{\pi_2} \\
        A        & \ar[l]_{\psi} A_1 \ar[r]^{\psi_2} & A_2}
\end{displaymath}
where $V = \widetilde{R}/\widetilde{N}$ is a vector space, $A = R/N$, $A_1 = \widetilde{R}/\widetilde{\Gamma}\widetilde{N}$, $A_2 = V/\Gamma$, with $\Gamma = \widetilde{\pi}_2\left(\widetilde{\Gamma}\right)$, the maps $\widetilde{\pi}$, $\pi$, $\pi_1$, $\widetilde{\pi}_2$, $\pi_2$ are the respective canonical projections.

Furthermore, let $\phi: R \to R$ be a continuous surjective endomorphism, $\widetilde{\phi}: \widetilde{R} \to \widetilde{R}$ be the endomorphism induced by $\phi$ on $\widetilde{R}$, $\phi_2: V \to V$ be the endomorphism induced by $\widetilde{\phi}$ on $V$, $\varphi: A \to A$ be the endomorphism induced by $\phi$ on $A$, and $\varphi_2: A_2 \to A_2$ be the endomorphism induced by $\phi_2$ on $A_2$. Then $\varphi$ is conjugated to $\varphi_2$ by the isomorphism $\psi \circ \psi_2^{-1}$ and that $d(\varphi_2|_{T(A_2)})_1$ is conjugated to $d\left(\phi_2|_{V_\Gamma}\right)_1$ by the isomorphism $d(\pi_2|_{V_\Gamma})_1$, where $V_\Gamma$ is the subspace of $V$ generated by $\Gamma$.
\end{proposicao}
\begin{proof} 
We have that $V = \widetilde{R}/\widetilde{N}$ is a vector space, since it is a connected and simply-connected abelian Lie group (see Proposition \ref{vector-space}). Since $\widetilde{\Gamma}$ is a subgroup of $Z\left(\widetilde{R}\right)$, it follows that $\widetilde{N}\widetilde{\Gamma} = \widetilde{\Gamma}\widetilde{N}$ and that
\begin{equation}
\ker (\pi \circ \widetilde{\pi}) = \ker \pi_1 = \ker (\pi_2 \circ \widetilde{\pi}_2)
\end{equation}
Thus we can define $\psi$ and $\psi_2$ in order that the above diagram commutes and they are isomorphisms. 

By definition, the induced isomorphisms are such that $\phi \circ \widetilde{\pi} = \widetilde{\pi} \circ \widetilde{\phi}$ and that $\varphi \circ \pi = \pi \circ \phi$. Thus we have that $\varphi \circ \pi \circ \widetilde{\pi} = \pi \circ \widetilde{\pi} \circ \widetilde{\phi}$ and, using the commutative diagram, we have that $\varphi \circ \psi \circ \pi_1 = \psi \circ \pi_1 \circ \widetilde{\phi}$. By a similar argument, we have that $\varphi_2 \circ \psi_2 \circ \pi_1 = \psi_2 \circ \pi_1 \circ \widetilde{\phi}$. Denoting by $\varphi_1: A_1 \to A_1$ the endomorphism induced by $\widetilde{\phi}$ on $A_1$, we have that $\varphi_1 \circ \pi_1 = \pi_1 \circ \widetilde{\phi}$ and thus we have $\varphi \circ \psi \circ \pi_1 = \psi \circ \varphi_1 \circ \pi_1$ and $\varphi_2 \circ \psi_2 \circ \pi_1 = \psi_2 \circ \varphi_1 \circ \pi_1$. These imply that $\varphi \circ \psi = \psi \circ \varphi_1$ and that $\varphi_2 \circ \psi_2 = \psi_2 \circ \varphi_1$, which imply that $\varphi \circ \psi \circ \psi_2^{-1} = \psi \circ \psi_2^{-1} \circ \varphi_2$.

Since $\phi_2 \circ \widetilde{\pi}_2 = \widetilde{\pi}_2 \circ \widetilde{\phi}$, it follows that $\Gamma = \widetilde{\pi}_2\left(\widetilde{\Gamma}\right)$ is invariant by $\phi_2$, which is a linear map. Thus $V_\Gamma$ is also invariant by $\phi_2$ and we can consider the restriction $\phi_2|_{V_\Gamma}$. Since $\pi_2 \circ \phi_2 = \varphi_2 \circ \pi_2$, it follows that $d(\varphi_2|_{\pi_2(V_\Gamma)})_1$ is conjugated to $d\left(\phi_2|_{V_\Gamma}\right)_1$ by the isomorphism $d(\pi_2|_{V_\Gamma})_1$. It remains to show that $\pi_2(V_\Gamma)$ is equal to $T(A_2)$. But if $W$ is a complement of $V_\Gamma$ in $V$, it follows that $\rho: A_2 \to \pi_2(V_\Gamma) \times W$, given by $\rho(\pi_2(v)) = (\pi_2(v_\Gamma), w)$, where $v = v_\Gamma + w$, $v_\Gamma \in V_\Gamma$, and $w \in W$, is well defined and a isomorphism of Lie groups. In fact, if $v' = v_\Gamma' + w'$, $v_\Gamma' \in V_\Gamma$, and $w' \in W$, we have that $\pi_2(v) = \pi_2(v')$ if an only if $v - v' \in \Gamma$ if an only if $v_\Gamma - v_\Gamma' + w - w' \in \Gamma \subset V_\Gamma$  if an only if $v_\Gamma - v_\Gamma' \in \Gamma$ and $w - w' = 0$ if an only if $\rho(\pi_2(v)) = \rho(\pi_2(v'))$. Since $\rho(\pi_2(V_\Gamma)) = T(\pi_2(V_\Gamma) \times W)$, it follows that $\pi_2(V_\Gamma)$ is equal to $T(A_2)$.
\end{proof}

The following lemma was stated and proved by Luiz San Martin in a personal communication.

\begin{lema}\label{lemma-Ad-finite-order}
Let $\mathfrak{a}$ be an abelian subalgebra of $\mathfrak{g} = \mathfrak{gl}(n,\mathbb{C})$ whose elements are semi-simple and let $\theta \in \mathrm{Gl}(n,\mathbb{C})$ be an element of the normalizer of $\mathfrak{a}$. Then the restriction of $\mbox{Ad}(\theta)$ to $\mathfrak{a}$ has finite order.
\end{lema}
\begin{proof}
By the assumptions, up to a conjugation, we can assume that $\mathfrak{a}$ is contained in the Cartan subalgebra $\mathfrak{h}$ of the diagonal matrices of $\mathfrak{g}$. Consider the centralizer of $\mathfrak{a}$ in $\mathfrak{g}$ given by
\begin{equation}
\mathfrak{z}(\mathfrak{a},\mathfrak{g}) 
= 
\{X \in \mathfrak{g}: [X,H] = 0, \mbox{ for all } H \in \mathfrak{a}\}
\end{equation}
Since $\theta$ normalizes $\mathfrak{a}$, it follows that $\theta$ normalizes $\mathfrak{z}(\mathfrak{a},\mathfrak{g})$. In fact, if $X \in \mathfrak{z}(\mathfrak{a},\mathfrak{g})$, it follows that
\begin{equation}
[\mbox{Ad}(\theta)X,H] = \mbox{Ad}(\theta)[X,\mbox{Ad}(\theta)^{-1}H] = \mbox{Ad}(\theta)0 = 0
\end{equation}
since $\mbox{Ad}(\theta)^{-1}H \in \mathfrak{a}$. On the other hand, from the root space decomposition of $\mathfrak{g}$ associated to $\mathfrak{h}$, it follows that 
\begin{equation}
\mathfrak{z}(\mathfrak{a},\mathfrak{g})
= 
\mathfrak{h} \oplus \sum \{\mathfrak{g}_\alpha : \alpha(H) = 0, \mbox{ for all } H \in \mathfrak{a}\}
\end{equation}
where $\alpha$ and $\mathfrak{g}_\alpha$ are the roots and the root spaces associated to $\mathfrak{h}$. This implies that $\mathfrak{z}(\mathfrak{a},\mathfrak{g})$ is given by matrices with the following block diagonal decomposition
\begin{equation}
\left(
\begin{array}{cccc}
X_1 & & \\
& \ddots & \\
& & X_k
\end{array}
\right)
\end{equation}
where $X_1,\ldots,X_k$ are arbitrary square matrices. Since $\theta$ normalizes $\mathfrak{z}(\mathfrak{a},\mathfrak{g})$, it follows that $\theta$ normalizes its center $\mathfrak{z}(\mathfrak{z}(\mathfrak{a},\mathfrak{g}))$, which is given by diagonal matrices of the following form 
\begin{equation}
\left(
\begin{array}{cccc}
z_1I_1 & & \\
& \ddots & \\
& & z_kI_k
\end{array}
\right)
\end{equation}
where $I_1,\ldots,I_k$ are identity matrices and $z_1,\ldots,z_k \in \mathbb{C}$. But the normalizer of $\mathfrak{z}(\mathfrak{z}(\mathfrak{a},\mathfrak{g}))$ is formed by matrices with the same block structure as above and by matrices which permute blocks of the same size, if they exist. Hence the restriction of $\mbox{Ad}(\theta)$ to $\mathfrak{z}(\mathfrak{z}(\mathfrak{a},\mathfrak{g}))$ has finite order, since matrices with the same block structure as above centralize $\mathfrak{z}(\mathfrak{z}(\mathfrak{a},\mathfrak{g}))$ and matrices which permute blocks of the same size have finite order. Thus the lemma follows, since $\mathfrak{a}$ is contained in $\mathfrak{z}(\mathfrak{z}(\mathfrak{a},\mathfrak{g}))$.
\end{proof}

Now we prove the main result of this section and the main technical result of the paper.

\begin{teorema}\label{teo-toral-finite-order}
Let $R$ be a connected solvable group, $\phi: R \to R$ be a continuous surjective endomorphism, $N$ be the nilpotent radical of $R$, and and $\varphi$ is the endomorphism induced by $\phi$ on $A = R/N$. If $N$ is simply-connected, then $\varphi|_{T(A)}$ has finite order.
\end{teorema}
\begin{proof}
By Proposition \ref{prop-diagram} and using its notation, we have that $d(\varphi|_{T(A)})_1$ is conjugated to $d(\varphi_2|_{T(A_2)})_1$, which is conjugated to $d\left(\phi_2|_{V_\Gamma}\right)_1$. Thus it is enough to show that $d\left(\phi_2|_{V_\Gamma}\right)_1$ has finite order.

Since the homomorphism $\widetilde{\pi}|_{\widetilde{N}}: \widetilde{N} \to N$ is a covering map and $N$ is simply-connected, it follows that $\widetilde{\pi}|_{\widetilde{N}}$ is an isomorphism. This implies that $\ker\left(\widetilde{\pi}|_{\widetilde{N}}\right) = \widetilde{N} \cap \widetilde{\Gamma} = \{1\}$ and that $\widetilde{\pi}_2|_{\widetilde{\Gamma}}: \widetilde{\Gamma} \to \Gamma$ is also an isomorphism, since $\ker\left(\widetilde{\pi}_2|_{\widetilde{\Gamma}}\right) = \widetilde{N} \cap \widetilde{\Gamma} = \{1\}$. Since $\Gamma$ is a discrete subgroup of the vector space $V = \widetilde{R}/\widetilde{N}$, there exists a linear independent set $\{v_1,\ldots,v_k\}$ such that $\Gamma =  \mathbb{Z}v_1 \oplus \cdots \oplus \mathbb{Z}v_k$, which implies that $\{v_1,\ldots,v_k\}$ is a basis of $V_\Gamma$. By Theorem \ref{fundamental-center}, fixing an abelian subalgebra $\mathfrak{h}$ of the Lie algebra $\mathfrak{r}$ of $\widetilde{R}$, we have that
\begin{equation}
Z\left(\widetilde{R}\right) = \exp\{X \in \mathfrak{h} : \mbox{the spectrum of } \mathrm{ad}(X) \subset 2\pi i \mathbb{Z} \}
\end{equation}
By Theorem \ref{fundamental-center}, we also have that $\exp(X) \in Z\left(\widetilde{R}\right)$ if and only if $\mathrm{ad}(X)$ is semi-simple in $\mathfrak{r}$ and the spectrum of $\mathrm{ad}(X)$ is a subset of $2\pi i \mathbb{Z}$. Since $\widetilde{\Gamma}$ is a subgroup of $Z\left(\widetilde{R}\right)$, there exists $\{X_1,\ldots,X_k\} \subset \mathfrak{h}$ such that $\widetilde{\pi}_2(\exp(X_i)) = v_i$, for all $i \in \{1,\ldots,k\}$. Denoting by $\mathfrak{h}_\Gamma$ the subspace of $\mathfrak{h}$ generated by $\{X_1,\ldots,X_k\}$, it follows that $\widetilde{H}_\Gamma = \exp(\mathfrak{h}_\Gamma)$ is a connected subgroup of $\widetilde{R}$. Since $\Gamma$ is generated by $\{v_1,\ldots,v_k\}$ and $\widetilde{\pi}|_{\widetilde{N}}$ is an isomorphism, it follows that $\widetilde{\Gamma}$ is generated by $\{\exp(X_1),\ldots,\exp(X_k)\}$. We have that $\widetilde{\Gamma} \subset \widetilde{H}_\Gamma$ and thus $\Gamma \subset \widetilde{\pi}_2\left(\widetilde{H}_\Gamma\right)$, which implies that $V_\Gamma \subset \widetilde{\pi}_2\left(\widetilde{H}_\Gamma\right)$, since they are subspaces of $V$. On the other hand, we have that $\dim\left(\widetilde{\pi}_2\left(\widetilde{H}_\Gamma\right)\right) \leq \dim\left(\widetilde{H}_\Gamma\right) \leq k = \dim\left(V_\Gamma\right)$, showing that $\widetilde{\pi}_2|_{\widetilde{H}_\Gamma}: \widetilde{H}_\Gamma \to V_\Gamma$ is an isomorphism, which implies that $\ker\left(\widetilde{\pi}_2|_{\widetilde{H}_\Gamma}\right) = \widetilde{N} \cap \widetilde{H}_\Gamma = \{1\}$. Since $\phi_2 \circ \widetilde{\pi}_2 = \widetilde{\pi}_2 \circ \widetilde{\phi}$ and $V_\Gamma$ is invariant by $\phi_2$, it follows that $\widetilde{H}_\Gamma\widetilde{N}$ is invariant by $\widetilde{\phi}$ and thus we can consider the following commutative diagram
\begin{displaymath}
    \xymatrix{
	\widetilde{H}_\Gamma\widetilde{N} \ar[d]_{\widetilde{\pi}_2} \ar[r]^{\widetilde{\phi}|_{\widetilde{H}_\Gamma\widetilde{N}}} &  \widetilde{H}_\Gamma\widetilde{N} \ar[d]^{\widetilde{\pi}_2} \\
        V_\Gamma \ar[r]_{\phi_2|_{V_\Gamma}} &  V_\Gamma }
\end{displaymath}
Since $\widetilde{H}_\Gamma \cap \widetilde{N} = \{1\}$, it follows that $\mathfrak{h}_\Gamma \cap \mathfrak{n} = \{0\}$, where $\mathfrak{n}$ is the Lie algebra of $\widetilde{N}$. Since $\mathfrak{n}$ contains the derived subalgebra of $\mathfrak{r}$, it follows that $\mathfrak{h}_\Gamma \oplus \mathfrak{n}$ is a subalgebra. We have that $\mathfrak{h}_\Gamma$ is compactly embedded in $\mathfrak{h}_\Gamma \oplus \mathfrak{n}$. In fact, for all $X \in \mathfrak{h}_\Gamma$, we have that $\mathrm{ad}(X)$ is semi-simple and its spectrum is a subset of $i \mathbb{R}$, since this is true for the basis $\{X_1,\ldots,X_k\}$ of $\mathfrak{h}_\Gamma$, which is an abelian subalgebra. Hence the same is true for the restriction $\mathrm{ad}(X)|_{\mathfrak{h}_\Gamma \oplus \mathfrak{n}}$, showing the claim. We can also consider the restriction $\mathrm{ad}(X)|_{\mathfrak{h}_\Gamma} = 0$ and the restriction $\mathrm{ad}(X)|_\mathfrak{n}$ which is semi-simple and whose spectrum is also a subset of $i \mathbb{R}$. Let $\widehat{\mathfrak{h}}$ be a maximal compactly embedded abelian subalgebra of $\mathfrak{h}_\Gamma \oplus \mathfrak{n}$ containing $\mathfrak{h}_\Gamma$. Since $\widetilde{\phi}\left(\widetilde{H}_\Gamma\widetilde{N}\right) \subset \widetilde{H}_\Gamma\widetilde{N}$, it follows that $d\widetilde{\phi}_1\left(\mathfrak{h}_\Gamma \oplus \mathfrak{n}\right) = \mathfrak{h}_\Gamma \oplus \mathfrak{n}$. By Theorem \ref{fundamental-center}, we have that $d\widetilde{\phi}_1\widehat{\mathfrak{h}}$ is a maximal compactly embedded abelian subalgebra of $\mathfrak{h}_\Gamma \oplus \mathfrak{n}$. We also have, by by Theorem \ref{fundamental-center}, that there exists an inner automorphism $\widehat{\psi}$ of $\mathfrak{h}_\Gamma \oplus \mathfrak{n}$ such that $\widehat{\psi} d\widetilde{\phi}_1\widehat{\mathfrak{h}} = \widehat{\mathfrak{h}}$. By the commutative diagram above, it follows that 
\begin{equation}
d\left(\widetilde{\pi}_2\right)_1 \circ d\left(\widetilde{\phi}|_{\widetilde{H}_\Gamma\widetilde{N}}\right)_1 = d\left(\phi_2|_{V_\Gamma}\right)_1 \circ d\left(\widetilde{\pi}_2\right)_1
\end{equation}
On the other hand, we have that $d\left(\widetilde{\pi}_2\right)_1 \circ \widehat{\psi} = d\left(\widetilde{\pi}_2\right)_1$. In fact, we have that $\widehat{\psi} = \mbox{e}^{\mathrm{ad}(Y_1)}\cdots\mbox{e}^{\mathrm{ad}(Y_l)}$, where $Y_1,\ldots, Y_l \in \mathfrak{h}_\Gamma \oplus \mathfrak{n}$ and, for all $Y \in \mathfrak{h}_\Gamma \oplus \mathfrak{n}$, we have that $d\left(\widetilde{\pi}_2\right)_1 \circ \mbox{e}^{\mathrm{ad}(Y)} = d\left(\widetilde{\pi}_2\right)_1$, since $d\left(\widetilde{\pi}_2\right)_1|_\mathfrak{n} = 0$ and $\mathrm{ad}(Y)\left(\mathfrak{h}_\Gamma \oplus \mathfrak{n}\right) \subset \mathfrak{n}$. Hence, putting $\widehat{\phi} = \widehat{\psi} \circ d\left(\widetilde{\phi}|_{\widetilde{H}_\Gamma\widetilde{N}}\right)_1$, it follows that $\widehat{\phi}\widehat{\mathfrak{h}} = \widehat{\mathfrak{h}}$ and that
\begin{equation}
d\left(\widetilde{\pi}_2\right)_1 \circ \widehat{\phi} = d\left(\phi_2|_{V_\Gamma}\right)_1 \circ d\left(\widetilde{\pi}_2\right)_1
\end{equation}
We also have that $\widehat{\mathfrak{h}} = \mathfrak{h}_\Gamma \oplus \left(\widehat{\mathfrak{h}} \cap \mathfrak{n}\right)$, since if $Z \in \widehat{\mathfrak{h}}$, then $Z = X + Y$, where $X \in \mathfrak{h}_\Gamma$ and $Y \in \mathfrak{n}$, so that $Y = Z - X \in \widehat{\mathfrak{h}}$. We also have that $\widehat{\mathfrak{h}} \cap \mathfrak{n} = \mathfrak{z}\left(\mathfrak{h}_\Gamma \oplus \mathfrak{n}\right)$, where $\mathfrak{z}\left(\mathfrak{h}_\Gamma \oplus \mathfrak{n}\right)$ is the center of $\mathfrak{h}_\Gamma \oplus \mathfrak{n}$. In fact, by Theorem \ref{fundamental-center}, we have that $\mathfrak{z}\left(\mathfrak{h}_\Gamma \oplus \mathfrak{n}\right) \subset \widehat{\mathfrak{h}}$, since $\widehat{\mathfrak{h}}$ is a maximal compactly embedded subalgebra of $\mathfrak{h}_\Gamma \oplus \mathfrak{n}$. We also have that $\mathfrak{z}\left(\mathfrak{h}_\Gamma \oplus \mathfrak{n}\right) \subset \mathfrak{n}$, since $\mathfrak{z}\left(\mathfrak{h}_\Gamma \oplus \mathfrak{n}\right)$ is an abelian ideal of $\mathfrak{r}$. In fact, if $Z \in \mathfrak{z}\left(\mathfrak{h}_\Gamma \oplus \mathfrak{n}\right)$, $X \in \mathfrak{r}$, and $Y \in \mathfrak{h}_\Gamma \oplus \mathfrak{n}$, we have that
\begin{equation}
[Y,[X,Z]] = [[Y,X],Z] + [X,[Y,Z]] = 0
\end{equation}
since $\mathfrak{n}$ contains the derived subalgebra of $\mathfrak{r}$, which shows that $[X,Z] \in \mathfrak{z}\left(\mathfrak{h}_\Gamma \oplus \mathfrak{n}\right)$. The previous inequalities imply that $\mathfrak{z}\left(\mathfrak{h}_\Gamma \oplus \mathfrak{n}\right) \subset \widehat{\mathfrak{h}} \cap \mathfrak{n}$. On the other hand, if $Y \in \widehat{\mathfrak{h}} \cap \mathfrak{n}$, we have that $\mathrm{ad}(Y)|_{\mathfrak{h}_\Gamma} = 0$ and that $\mathrm{ad}(Y)|_{\mathfrak{n}}$ is nilpotent, which implies that $\mathrm{ad}(Y)|_{\mathfrak{h}_\Gamma \oplus \mathfrak{n}}$ is nilpotent. Hence the spectrum of $\mathrm{ad}(Y)|_{\mathfrak{h}_\Gamma \oplus \mathfrak{n}}$ is zero and thus is a subset of $2\pi i \mathbb{Z}$. By Theorem \ref{fundamental-center}, it follows that $\exp(Y)$ belongs to the center of the connected subgroup generated by $\mathfrak{h}_\Gamma \oplus \mathfrak{n}$. Thus it follows that $\mathrm{ad}(Y)|_{\mathfrak{h}_\Gamma \oplus \mathfrak{n}}$ is semi-simple. Since $\mathrm{ad}(Y)|_{\mathfrak{h}_\Gamma \oplus \mathfrak{n}}$ is also nilpotent, this implies that $\mathrm{ad}(Y)|_{\mathfrak{h}_\Gamma \oplus \mathfrak{n}} = 0$, showing that $Y \in \mathfrak{z}\left(\mathfrak{h}_\Gamma \oplus \mathfrak{n}\right)$.

Now denote by $P: \widehat{\mathfrak{h}} \to \mathfrak{h}_\Gamma$ the linear projection associated to the decomposition $\widehat{\mathfrak{h}} = \mathfrak{h}_\Gamma \oplus \left(\widehat{\mathfrak{h}} \cap \mathfrak{n}\right) = \mathfrak{h}_\Gamma \oplus \mathfrak{z}\left(\mathfrak{h}_\Gamma \oplus \mathfrak{n}\right)$. We have that 
$d\left(\widetilde{\pi}_2\right)_1|_{\mathfrak{h}_\Gamma} \circ P = d\left(\widetilde{\pi}_2\right)_1|_{\widehat{\mathfrak{h}}}$, since $d\left(\widetilde{\pi}_2\right)_1|_\mathfrak{n} = 0$. Putting $T = P \circ \widehat{\phi}|_{\mathfrak{h}_\Gamma}: \mathfrak{h}_\Gamma \to \mathfrak{h}_\Gamma$, we have that
\begin{eqnarray*}
d\left(\widetilde{\pi}_2\right)_1|_{\mathfrak{h}_\Gamma} \circ T 
& = & d\left(\widetilde{\pi}_2\right)_1|_{\mathfrak{h}_\Gamma} \circ P \circ \widehat{\phi}|_{\mathfrak{h}_\Gamma} \\
& = & d\left(\widetilde{\pi}_2\right)_1|_{\widehat{\mathfrak{h}}} \circ \widehat{\phi}|_{\mathfrak{h}_\Gamma} \\
& = & d\left(\phi_2|_{V_\Gamma}\right)_1 \circ d\left(\widetilde{\pi}_2\right)_1|_{\mathfrak{h}_\Gamma}
\end{eqnarray*}
Since $\widetilde{\pi}_2|_{\widetilde{H}_\Gamma}: \widetilde{H}_\Gamma \to V_\Gamma$ is an isomorphism, we have that $d\left(\phi_2|_{V_\Gamma}\right)_1$ is conjugated to $T$ and thus it is enough to show that $T$ has finite order. Now, considering the homomorphism $\rho: \mathfrak{h}_\Gamma \to \mathfrak{gl}(\mathfrak{n})$, given by $\rho(X) = \mathrm{ad}(X)|_{\mathfrak{n}}$, we have that $T$ is conjugated to $\mbox{Ad}(\theta)$ restricted to $\mathfrak{a}$, where $\theta = \widehat{\phi}|_{\mathfrak{n}}$ and $\mathfrak{a}$ is the image of $\rho$. In fact, we have that $\mbox{Ad}(\theta) \circ \rho = \rho \circ T$, since
\begin{eqnarray*}
\left(\mbox{Ad}(\theta) \circ \rho\right)(X)Y
& = & \theta(\rho(X))\theta^{-1}Y \\
& = & \theta[X,\theta^{-1}Y] \\
& = & \widehat{\phi}[X,\theta^{-1}Y] \\
& = & \left[\widehat{\phi}X,\widehat{\phi}\theta^{-1}Y\right] \\
& = & [TX,Y] \\
& = & \rho(TX) Y \\
& = & \left(\rho \circ T\right)(X)Y
\end{eqnarray*}
where $X \in \mathfrak{h}_\Gamma$, $Y \in \mathfrak{n}$, and we used that $\widehat{\phi}\widehat{\mathfrak{h}} = \widehat{\mathfrak{h}}$. It remains to prove that $\rho$ is an isomorphism. In fact, if $\rho(X) = \mathrm{ad}(X)|_\mathfrak{n} = 0$, then $X \in \mathfrak{z}\left(\mathfrak{h}_\Gamma \oplus \mathfrak{n}\right)$, since $X \in \mathfrak{h}_\Gamma$, which implies that $X = 0$, since $\mathfrak{h}_\Gamma  \cap \mathfrak{z}\left(\mathfrak{h}_\Gamma \oplus \mathfrak{n}\right) = 0$. Since $T$ and $\mbox{Ad}(\theta)$ restrict to $\mathfrak{a}$ are conjugated, the result follows from Lemma \ref{lemma-Ad-finite-order}, since each $\mathrm{ad}(X)|_\mathfrak{n} \in \mathfrak{a}$ is semi-simple and its spectrum is a subset of $i \mathbb{R}$. 
\end{proof}

\section{Li-Yorke pairs}\label{li-yorke}

In this section, we prove some useful results about the existence and nonexistence of Li-Yorke pairs. Given a continuous map $T: X \to X$ of a topological space $X$, a pair $(a,b) \in X \times X$ is called a \emph{Li-Yorke pair} when $a \neq b$ and there exist $c \in X$, $n_l \rightarrow \infty$ and $n_k \rightarrow \infty$ such that
\begin{equation}
\left(T^{n_l}(a),T^{n_l}(b)\right) \to (a,b)
\qquad \mbox{and} \qquad
\left(T^{n_k}(a),T^{n_k}(b)\right) \to (c,c)
\end{equation}
The main result shows that if the maximal torus of the center of a Lie group is trivial and the restriction of a endomorphism to the connected component of the identity is surjective, then there exists a positive power of this endomorphism that has no Li-Yorke pair. We start with a general result about Li-Yorke pairs and quotient of Lie groups.

\begin{proposicao}\label{prop-li-yorke-G-H}
Let $\phi: G \to G$ be a continuous endomorphism and $H$ be a closed subgroup of the Lie group $G$ such that $\phi(H) \subset H$. 
Consider the following commutative diagram
\begin{displaymath}
    \xymatrix{
        G \ar[r]^\phi \ar[d]_\pi & G \ar[d]^\pi \\
        G/H \ar[r]_\varphi       & G/H }
\end{displaymath}
where $\pi$ is the canonical projection and $\varphi$ has no Li-Yorke pair.  Then $\phi$ has a Li-Yorke pair if and only if $\phi|_H$ has a Li-Yorke pair.
\end{proposicao}
\begin{proof}
Of course, if $\phi|_H$ has a Li-Yorke pair, then $\phi$ has a Li-Yorke pair. Now assume $\phi|_H$ has no Li-Yorke pair. Let $a,b,c \in G$ such that
\begin{equation}
\left(\phi^{n_l}(a),\phi^{n_l}(b)\right) \to (a,b)
\qquad \mbox{and} \qquad
\left(\phi^{n_k}(a),\phi^{n_k}(b)\right) \to (c,c)
\end{equation}
Using the commutative diagram, it follows that
\begin{equation}
\left(\varphi^{n_l}(\pi(a)),\varphi^{n_l}(\pi(b))\right) \to (\pi(a),\pi(b))
\end{equation}
and that
\begin{equation}
\left(\varphi^{n_k}(\pi(a)),\varphi^{n_k}(\pi(b))\right) \to (\pi(c),\pi(c))
\end{equation}
Since $\varphi$ has no Li-Yorke pair, it follows that $\pi(a) = \pi(b)$. Hence $a = bh$, for some $h \in H$, and 
\begin{eqnarray*}
\phi^{n_l}(h) & = & \left(b^{-1}\phi^{n_l}(b)\right)^{-1}\left(b^{-1}\phi^{n_l}(b)\right)\phi^{n_l}(h) \\
& = & \left(b^{-1}\phi^{n_l}(b)\right)^{-1}b^{-1}\phi^{n_l}(a) \\
& \to & b^{-1}a \\
& = & h
\end{eqnarray*}
On the other hand, considering a left-invariant distance in $G$, we have that
\begin{eqnarray*}
d\left(\phi^{n_k}(h), 1\right) & = & d\left(\phi^{n_k}(b)\phi^{n_k}(h), \phi^{n_k}(b)\right) \\
& = & d\left(\phi^{n_k}(a), \phi^{n_k}(b)\right) \\
& \to & 0
\end{eqnarray*}
Thus we have that 
\begin{equation}
\left(\phi^{n_l}(h),\phi^{n_l}(1)\right) \to (h,1)
\qquad \mbox{and} \qquad
\left(\phi^{n_k}(h),\phi^{n_k}(1)\right) \to (1,1)
\end{equation}
Since $\phi|_H$ has no Li-Yorke pair, it follows that $h = 1$, showing that $\phi$ has no Li-Yorke pair.
\end{proof}

The following result reduces the problem of showing the nonexistence of Li-Yorke pairs to considering the restriction of an endomorphism to the solvable radical of the Lie group.

\begin{proposicao}\label{prop-phi-r-no-li-yorke-pair}
Let $G$ be a Lie group, $\phi: G \to G$ be a continuous endomorphism such that $\phi(G_0) = G_0$ and $R$ be the solvable radical of $G_0$. If $\phi|_R$ has no Li-Yorke pair, then $\phi^n$ has no Li-Yorke pair for some $n > 0$. 
\end{proposicao}
\begin{proof}
Since $\phi(R) \subset R$, we can consider the following commutative diagram
\begin{displaymath}
    \xymatrix{
        G_0 \ar[r]^{\phi|_{G_0}} \ar[d]_\pi & G_0 \ar[d]^\pi \\
        S \ar[r]_\varphi       & S }
\end{displaymath}
where $S = G_0/R$ is a semi-simple Lie group, $\pi$ is the canonical projection and $\varphi$ is the endomorphism induced by $\phi$ on $S$. Since $\phi(G_0) = G_0$, it follows that $\varphi(S) = S$, which implies that $\varphi(Z(S)) \subset Z(S)$, where $Z(S)$ is the center of $S$. Hence we can also consider the following commutative diagram
\begin{displaymath}
    \xymatrix{
        S \ar[r]^\varphi \ar[d]_\rho & S \ar[d]^\rho \\
        S/Z(S) \ar[r]_\psi      & S/Z(S) }
\end{displaymath}
where $\rho$ is the canonical projection and $\psi$ is the endomorphism induced by $\varphi$ on $S/Z(S)$. We have that $S/Z(S)$ is isomorphic to the adjoint group of $S$, and hence it is a linear semi-simple Lie group. Thus $\psi^n$ is conjugated to the restriction of a linear map for some $n > 0$ (see Proposition 10 of \cite{caldas-patrao:endomorfismos}) and thus $\psi^n$ has no Li-Yorke pair. We also have that $Z(S)$ is discrete, and hence $\varphi^n|_{Z(S)}$ has no Li-Yorke pair. By Proposition \ref{prop-li-yorke-G-H}, it follows that $\varphi^n$ has no Li-Yorke pair. Since $\phi|_R$ has no Li-Yorke pair, it follows that $\phi^n|_R$ has no Li-Yorke pair. Hence, by Proposition \ref{prop-li-yorke-G-H}, it follows that $\phi^n|_{G_0}$ has no Li-Yorke pair. Finally, since $\phi(G_0) = G_0$, we can consider the following commutative diagram
\begin{displaymath}
    \xymatrix{
        G \ar[r]^\phi \ar[d]_\pi & G \ar[d]^\pi \\
        C \ar[r]_{\gamma}       & C }
\end{displaymath}
where $C = G/G_0$ is the group of the connected components of $G$, $\pi$ is the canonical projection and $\gamma$ is the endomorphism induced by $\phi$ on $C$. Since $C$ is a discrete group, it follows that $\gamma$ has no Li-Yorke pair. Hence $\gamma^n$ has no Li-Yorke pair and, since $\phi^n|_{G_0}$ has no Li-Yorke pair, by Proposition \ref{prop-li-yorke-G-H}, it follows that $\phi^n$ has no Li-Yorke pair.
\end{proof}

The next result is a consequence of a previous result of this section and the main result of the previous section.

\begin{proposicao}\label{prop-r-n-simly-connected-no-li-yorke-pair}
Let $R$ be a connected solvable group, $\phi: R \to R$ be a continuous surjective endomorphism and $N$ be the nilpotent radical of $R$. If $N$ is simply-connected, then $\phi$ has no Li-Yorke pair. 
\end{proposicao}
\begin{proof}
Since $\phi(N) \subset N$, we can consider the following commutative diagram
\begin{displaymath}
    \xymatrix{
        R \ar[r]^\phi \ar[d]_\pi & R \ar[d]^\pi \\
        A \ar[r]_\varphi       & A }
\end{displaymath}
where $A = R/N$ is an abelian Lie group, $\pi$ is the canonical projection and $\varphi$ is the endomorphism induced by $\phi$ on $A$. Since $N$ is a nilpotent and simply-connected Lie group, we have that $\phi|_N$ is conjugated to a linear map and thus has no Li-Yorke pair. Since $\varphi(T(A)) \subset T(A)$, we can also consider the following commutative diagram
\begin{displaymath}
    \xymatrix{
        A \ar[r]^\varphi \ar[d]_\rho & A \ar[d]^\rho \\
        A/T(A) \ar[r]_\psi      & A/T(A) }
\end{displaymath}
where $T(A)$ is the toral component of $A$, $\rho$ is the canonical projection and $\psi$ is the endomorphism induced by $\varphi$ on $A/T(A)$. Since $A/T(A)$ is an abelian and simply-connected Lie group, we have that $\psi$ is conjugated to a linear map and thus has no Li-Yorke pair. By Theorem \ref{teo-toral-finite-order}, it follows that $\varphi|_{T(A)}$ has finite order an thus has no Li-Yorke pair. Hence, by Proposition \ref{prop-li-yorke-G-H}, it follows that $\varphi$ has no Li-Yorke pair and also that $\phi$ has no Li-Yorke pair.
\end{proof}

Now we obtain the main result of this section, which is fundamental to the computation of the topological entropy in the next section.

\begin{corolario}\label{cor-tg-trivial-no-li-yorke-pair}
Let $G$ be a Lie group and $\phi: G \to G$ be a continuous endomorphism such that $\phi(G_0) = G_0$. If $T(G)$ is trivial, then $\phi^n$ has no Li-Yorke pair for some $n > 0$.  
\end{corolario}
\begin{proof}
By Proposition \ref{prop-toral-components-coincide}, it follows that $T(R)$ is trivial and thus $R \simeq R/T(R)$. Hence it follows from Proposition \ref{prop-toral-components-trivial} that the nilpotent radical of $R$ is simply-connected. Since $\phi(G_0) = G_0$, it also follows that $\phi|_R$ is a continuous surjective endomorphism. By Proposition \ref{prop-r-n-simly-connected-no-li-yorke-pair}, it follows that $\phi|_R$  has no Li-Yorke pair. The result then follows from Proposition \ref{prop-phi-r-no-li-yorke-pair}.
\end{proof}

\section{Topological entropy}\label{topological}

In this final section, we prove the main results of this paper, about the computation of the topological entropy of endomorphisms of Lie groups. We start with a partial generalization to noncompact spaces of a result due to Bowen (see Theorem 19 of \cite{bowen-endomorphism}).

\begin{proposicao}\label{prop-bowen-generalizado}
Let $X$ and $Y$ be locally compact metrizable topological spaces. If the following diagram commutes
\begin{displaymath}
    \xymatrix{
        X \ar[r]^\phi \ar[d]_\pi & X \ar[d]^\pi \\
        Y \ar[r]_\varphi       & Y }
\end{displaymath}
and the maps $\pi$, $\phi$, and $\varphi$ are continuous proper maps, then
\begin{equation}
h(\phi) \leq h(\varphi) + \sup_{y \in Y} h_{d_X}\left(\phi,\pi^{-1}(y)\right)
\end{equation}
where $d_X$ is a distance induced by a distance of the one-point compactification. Moreover, if $\pi: X \to Y$ is a $T$-principal bundle, where $T$ is a compact group, and there exists $\tau: T \to T$ such that
\begin{equation}
\phi(xt) = \phi(x)\tau(t)
\end{equation}
for all $x \in X$ and $t \in T$, then
\begin{equation}
h(\phi) \leq h(\varphi) + h(\tau)
\end{equation}
\end{proposicao}
\begin{proof}
We can consider $\widetilde{X}$ and $\widetilde{Y}$, the one-point compactification of $X$ and $Y$, and the following diagram
\begin{displaymath}
    \xymatrix{
        \widetilde{X} \ar[r]^{\widetilde{\phi}} \ar[d]_{\widetilde{\pi}} & \widetilde{X} \ar[d]^{\widetilde{\pi}} \\
        \widetilde{Y} \ar[r]_{\widetilde{\varphi}}       & \widetilde{Y} }
\end{displaymath}
where $\widetilde{\pi}$, $\widetilde{\phi}$, and $\widetilde{\varphi}$ are continuous maps, given by the natural extension of the continuous proper maps $\pi$, $\phi$, and $\varphi$. Now we can apply Bowen's result for compact spaces and get
\begin{equation}
h(\widetilde{\phi}) \leq h(\widetilde{\varphi}) + \sup_{y \in \widetilde{Y}} h_{d_{\widetilde{X}}}\left(\widetilde{\phi},\widetilde{\pi}^{-1}(y)\right)
\end{equation}
where $d_{\widetilde{X}}$ is a distance of $\widetilde{X}$.
Since $\widetilde{\pi}^{-1}(\infty) = \infty$ and $\widetilde{\phi}^{-1}(\infty) = \infty$, it follows that
\begin{equation}
h_{d_{\widetilde{X}}}\left(\widetilde{\phi},\widetilde{\pi}^{-1}(\infty)\right) = 0
\end{equation}
and thus
\begin{equation}
\sup_{y \in \widetilde{Y}} h_{d_{\widetilde{X}}}\left(\widetilde{\phi},\widetilde{\pi}^{-1}(y)\right) 
=
\sup_{y \in Y} h_{d_X}\left(\phi,\pi^{-1}(y)\right)
\end{equation}
where $d_X$ is a distance induced by $d_{\widetilde{X}}$. The result then follows, since $h(\widetilde{\phi}) = h(\phi)$ and $h(\widetilde{\varphi}) = h(\varphi)$ (see Proposition 2.3 of \cite{patrao:entropia}).

For the second part of the proposition, let $p: X \times T \to X$ be the continuous free action associated to the $T$-principal bundle $\pi: X \to Y$. Define $\widetilde{p}: \widetilde{X} \times T \to \widetilde{X}$ an extension of $p$ putting $\infty t = \infty$, for all $t \in T$. We claim that $\widetilde{p}$ is a continuous map. In fact, let $x_n \to \infty$ and $t_n \to t$. If $x_nt_n$ does not converge to $\infty$, there exist $\varepsilon > 0$ and a subsequence $x_{n_k}t_{n_k} \in K$, where $K = \widetilde{X} \backslash B_{d_{\widetilde{X}}}(\infty, \varepsilon)$. Since $K$ is compact, there exist $x \in K$ and $x_{n_{k_j}}t_{n_{k_j}} \to x$. Hence
\begin{equation}
x_{n_{k_j}} = x_{n_{k_j}}t_{n_{k_j}}t_{n_{k_j}}^{-1} \to xt^{-1}
\end{equation}
which contradicts $x_n \to \infty$, showing that $\widetilde{p}$ is continuous. Thus, given a distance $d_T$ of $T$ and $\varepsilon > 0$, there exists $\delta > 0$ such that, if $d_T(t,t') < \delta$, then $d_{\widetilde{X}}(\widetilde{x}t,\widetilde{x}t') < \varepsilon$, for all $\widetilde{x} \in \widetilde{X}$. It follows that, if $d_T(t,t') < \delta$, then $d_X(xt,xt') < \varepsilon$, for all $x \in X$. Now, given $y \in Y$, consider $x \in \pi^{-1}(y)$. It follows that $d_T(\tau^k(t),\tau^k(t')) < \delta$ implies that
\begin{equation}
d_X(\phi^k(xt),\phi^k(xt)) = d_X(\phi^k(x)\tau^k(t),\phi^k(x)\tau^k(t')) < \varepsilon
\end{equation}
Hence, if $T_n$ is an $(n,\delta)$-spanning set of $T$ for $\tau$, then $xT_n$ is an $(n,\varepsilon)$-spanning set of $\pi^{-1}(y) = xT$ for $\phi$. Thus $s_n^\phi\left(\varepsilon,\pi^{-1}(y)\right) \leq s_n^\tau(\delta,T)$, where $s_n^T(\varepsilon,Y)$ is the smallest cardinality of a $(n,\varepsilon)$-spanning set of $Y \subset X$ for the continuous map $T: X \to X$. This implies that 
\begin{equation}
h_{d_X}\left(\phi,\pi^{-1}(y)\right) \leq h_{d_T}(\tau) = h(\tau)
\end{equation}
where the last equality follows, since $T$ is compact. Now the second part of the proposition follows directly from the first one.
\end{proof}

Now we prove the main result of the paper, which shows that the computation of the topological entropy of a continuous endomorphism of a Lie group reduces to the classical formula for the topological entropy of a continuous endomorphism of a torus.

\begin{teorema}\label{theo-entropia-endomorphism}
Let $G$ be a Lie group and $\phi: G \to G$ be a continuous endomorphism. Then we have that
\begin{equation}
h\left(\phi\right) = h\left(\phi|_{T(G_\phi)}\right)
\end{equation}
where $G_\phi$ is the maximal connected subgroup of $G$ such that $\phi(G_\phi) = G_\phi$, and $T(G_\phi)$ is the maximal torus in the center of $G_\phi$. In particular, there always exists a $\phi$-invariant probability measure of maximal entropy.
\end{teorema}
\begin{proof}
By an argument similar to the proof of Lemma 4.1 of \cite{caldas-patrao:entropia}, there is $n > 0$ such that, if $H = \phi^n(G)$ and $H_0 = \phi^n(G_0) = G_\phi$, then $h(\phi) = h\left(\phi|_H\right)$ and $\phi(G_\phi) = G_\phi$. It follows that $G_\phi$ is the maximal connected subgroup of $G$ such that $\phi(G_\phi) = G_\phi$. Thus, we can consider the following commutative diagram
\begin{displaymath}
    \xymatrix{
        H \ar[r]^{\phi|_H} \ar[d]_\pi & H \ar[d]^\pi \\
        H/H_0 \ar[r]_\varphi       & H/H_0 }
\end{displaymath}
where $\pi$ is the canonical projection and $\varphi$ is the endomorphism induced by $\phi$ on $H/H_0$. Since $\pi$ is a continuous map, it follows that $\pi\left(\mathcal{R}_{\phi|_H}\right) \subset \mathcal{R}_\varphi$. Since $H/H_0$ is discrete, it follows that $\mathcal{R}_\varphi = \mathcal{P}_\varphi$, where $\mathcal{P}_\varphi$ is the set of periodic points of $\varphi$. Thus we have that $\mathcal{R}_{\phi|_H} \subset \pi^{-1}\left(\mathcal{P}_\varphi\right)$, which implies that $\overline{\mathcal{R}_{\phi|_H}} \subset \pi^{-1}\left(\mathcal{P}_\varphi\right)$, since $\mathcal{P}_\varphi$ is a closed set in $H/H_0$. Hence it follows that $h\left(\phi|_H\right) = h\left(\phi|_{\pi^{-1}\left(\mathcal{P}_\varphi\right)}\right)$. We have that $\mathcal{P}_\varphi$ is a subgroup of $H/H_0$. In fact, if $x, y \in \mathcal{P}_\varphi$, we have that $\varphi^l\left(x\right) = x$ and also that $\varphi^m\left(y\right) = y$, so that $\varphi^{lm}\left(xy\right) = \varphi^{lm}\left(x\right)\varphi^{lm}\left(y\right) = xy$. We also have that $\varphi|_{\mathcal{P}_\varphi}$ is an isomorphism. In fact, that $\varphi|_{\mathcal{P}_\varphi}$ is onto is immediate. Let $x \in \mathcal{P}_\varphi$ be such that $\varphi\left(x\right) = 1$. Since $\varphi^l\left(x\right) = x$, it follows that
\begin{equation}
x = \varphi^{l-1}\left(\varphi\left(x\right)\right) = \varphi^{l-1}\left(1\right) = 1
\end{equation}
showing that $\varphi|_{\mathcal{P}_\varphi}$ is injective. Finally, we have that $\phi|_{\pi^{-1}\left(\mathcal{P}_\varphi\right)}$ is a proper map, which is equivalent to show that the kernel of $\phi|_{\pi^{-1}\left(\mathcal{P}_\varphi\right)}$ is finite. Let $h \in \pi^{-1}\left(\mathcal{P}_\varphi\right)$ be such that $\phi\left(h\right) = 1$. It follows that $\pi\left(h\right) \in \mathcal{P}_\varphi$ and also that $\varphi\left(\pi\left(h\right)\right) = 1$. Since $\varphi|_{\mathcal{P}_\varphi}$ is an isomorphism, it follows that $\pi\left(h\right) = 1$, which implies that $h \in H_0 = G_\phi$. Hence the kernel of $\phi|_{\pi^{-1}\left(\mathcal{P}_\varphi\right)}$ is contained in the kernel of $\phi|_{G_\phi}$, which is finite, since $\phi(G_\phi) = G_\phi$ and $G_\phi$ is connected. Since $\phi\left(T(G_\phi)\right) \subset T(G_\phi)$, we can consider the following commutative diagram
\begin{displaymath}
    \xymatrix{
        \pi^{-1}\left(\mathcal{P}_\varphi\right) \ar[r]^{\phi|_{\pi^{-1}\left(\mathcal{P}_\varphi\right)}} \ar[d]_\pi & \pi^{-1}\left(\mathcal{P}_\varphi\right) \ar[d]^\pi \\
        \pi^{-1}\left(\mathcal{P}_\varphi\right)/T(G_\phi) \ar[r]_\psi       & \pi^{-1}\left(\mathcal{P}_\varphi\right)/T(G_\phi) }
\end{displaymath}
where $\pi$ is the canonical projection and $\psi$ is the endomorphism induced by $\phi$ on $\pi^{-1}\left(\mathcal{P}_\varphi\right)/T(G_\phi)$. We have that the connected component of the identity of $\pi^{-1}\left(\mathcal{P}_\varphi\right)/T(G_\phi)$ is given by $G_\phi/T(G_\phi)$, which is connected, since $G_\phi$ is connected. By Proposition \ref{prop-toral-components-trivial}, we have that $T\left(G_\phi/T(G_\phi)\right)$ is trivial. And we also have that $\psi\left(G_\phi/T(G_\phi)\right) = G_\phi/T(G_\phi)$, since $\phi(G_\phi) = G_\phi$. Hence, by Corollary \ref{cor-tg-trivial-no-li-yorke-pair}, it follows that $\psi^n$ has no Li-Yorke pair for some $n > 0$. Thus $h\left(\psi^n\right) = 0$ (see Proposition 6 of \cite{caldas-patrao:endomorfismos}), which implies that $h\left(\psi\right) = 0$. By Proposition \ref{prop-bowen-generalizado}, considering $\tau = \phi|_{T(G_\phi)}$, it follows that
\begin{eqnarray*}
h(\phi) & = & h\left(\phi|_H\right) \\
& = & h\left(\phi|_{\pi^{-1}\left(\mathcal{P}_\varphi\right)}\right) \\
& \leq & h(\psi) + h\left(\phi|_{T(G_\phi)}\right) \\
& = & h\left(\phi|_{T(G_\phi)}\right) \\
& \leq & h\left(\phi\right) 
\end{eqnarray*}
This implies that $h\left(\phi\right) = h\left(\phi|_{T(G_\phi)}\right)$. The last assertion holds, since there always exists a $\phi|_{T(G_\phi)}$-invariant probability measure of maximal entropy, as proved by Bowen in \cite{bowen-endomorphism}.
\end{proof}

Next we state an immediate corollary of the previous theorem.

\begin{corolario}\label{cor-entropia-endomorphism}
Let $G$ be a Lie group and $\phi: G \to G$ be a continuous automorphism. Then we have that
\begin{equation}
h\left(\phi\right) = h\left(\phi|_{T(G)}\right)
\end{equation}
where $T(G)$ is the maximal torus in the center of $G$. In particular, if $G$ is simply-connected, then $h\left(\phi\right) = 0$.
\end{corolario}

The following result was already proved by Andres Koropecki in \cite{koropecki}.

\begin{proposicao}\label{prop-torus-h-0-no-li-yorke-pair}
Let $T$ be a torus and $\phi: T \to T$ be a continuous and surjective endomorphism. Then $h(\phi) > 0$ if and only if $\phi^n$ has a Li-Yorke pair for each $n > 0$. 
\end{proposicao}
\begin{proof}
If $h(\phi) > 0$, then $h\left(\phi^n\right) = nh(\phi) > 0$ for each $n > 0$. Hence $\phi^n$ has a Li-Yorke pair for each $n > 0$ (see Proposition 6 of \cite{caldas-patrao:endomorfismos}). On the other hand, assume that $h(\phi) = 0$. By Bowen's formula, it follows all the eigenvalues of $d\phi_1$ have absolute value equal to one. Hence the determinant of $d\phi_1$ has absolute value equal to one. Since the determinant of $d\phi_1$ is an integer, this implies that it is equal to $\pm 1$ and thus that $\phi$ is an automorphism. By induction on the dimension $\dim(T)$ of $T$, we prove that there is $n > 0$ such that $\phi^n$ has no Li-Yorke pair. If the $\dim(T) = 1$, then $\phi$ is the identity map or the inversion map, and thus we have the result. Now assume that $\dim(T) = d > 1$. By a Corollary of Theorem 24.5 of \cite{denker-et-al}, we have that $\phi$ is not ergodic. Then Proposition 24.1 of \cite{denker-et-al} implies that $A = d\phi_1 \in \mbox{Gl}(n,\mathbb{Z})$ has an eigenvalue which is a root of unity. Thus there is $m > 0$ and $v \in \mathbb{Z}^d$ such that $A^m v = v$. In fact, we have that $(A^m - I)v = 0$ is a linear system with rational coefficients and a nontrivial solution. Consider $S$ the connected subgroup of $T$ generated by $v$. Then $S$ is closed, its dimension $\dim(S) = 1$, and $\phi^m(S) = S$. Thus we can consider the following commutative diagram
\begin{displaymath}
    \xymatrix{
        T \ar[r]^{\phi^m} \ar[d]_\pi & T \ar[d]^\pi \\
        T/S \ar[r]_\varphi       & T/S }
\end{displaymath}
where $\pi$ is the canonical projection and $\varphi$ is the endomorphism induced by $\phi^m$ on the torus $T/S$. Since $\dim(T/S) = d-1$, by the induction hypothesis, there is $k > 0$ such that $\varphi^k$ has no Li-Yorke pair. Thus $\varphi^{k}$ and $\phi^{mk}|_S$ have no Li-Yorke pair. Hence, by Proposition \ref{prop-li-yorke-G-H}, it follows that $\phi^{mk}$ has no Li-Yorke pair.
\end{proof}

We conclude this paper generalizing the previous result to endomorphisms of general Lie groups.

\begin{proposicao}
Let $G$ be a Lie group and $\phi: G \to G$ be a continuous endomorphism such that $\phi(G_0) = G_0$. Then $h(\phi) > 0$ if and only if $\phi^n$ has a Li-Yorke pair for each $n > 0$. 
\end{proposicao}
\begin{proof}
If $h(\phi) > 0$, then $h\left(\phi^n\right) = nh(\phi) > 0$ for each $n > 0$. Hence $\phi^n$ has a Li-Yorke pair for each $n > 0$. On the other hand, assume that $h(\phi) = 0$. By Theorem \ref{theo-entropia-endomorphism}, it follows that $h\left(\phi|_{T(G)}\right) = 0$. By Proposition \ref{prop-torus-h-0-no-li-yorke-pair}, there is $m > 0$ such that $\phi^m|_{T(G)}$ has no Li-Yorke pair. We can consider again the following commutative diagram
\begin{displaymath}
    \xymatrix{
        G \ar[r]^\phi \ar[d]_\pi & G \ar[d]^\pi \\
        G/T(G) \ar[r]_\varphi       & G/T(G) }
\end{displaymath}
where $\pi$ is the canonical projection and $\varphi$ is the endomorphism induced by $\phi$ on $G/T(G)$. By Proposition \ref{prop-toral-components-trivial}, we have that $T\left(G/T(G)\right)$ is trivial. Hence, by Corollary \ref{cor-tg-trivial-no-li-yorke-pair}, it follows that $\varphi^k$ has no Li-Yorke pair for some $k > 0$. Thus $\varphi^{mk}$ and $\phi^{mk}|_{T(G)}$ have no Li-Yorke pair. Hence, by Proposition \ref{prop-li-yorke-G-H}, it follows that $\phi^{mk}$ has no Li-Yorke pair.
\end{proof}


\begin{thebibliography}{99}
\bibitem{akm}  R. Adler, A. Konheim and H. MacAndrew: \
\textit{Topological entropy}. Trans. Americ. Math Soc.
\textbf{114} (1965), 309-319.
  
\bibitem{bowen-endomorphism}  R. Bowen: \
\textit{Entropy for group endomorphisms and homogeneous spaces}. Trans. Americ.
Math Soc. \textbf{153} (1971), 401-414.

\bibitem{caldas-patrao:endomorfismos} A. Caldas and M. Patr\~ao: \
\emph{Entropy of Endomorphisms of Lie Groups}.
Discrete and Continuous Dynamical Systems \textbf{33} (2013), 1351-1363.

\bibitem{caldas-patrao:entropia} A. Caldas and M. Patr\~ao: \
\emph{Entropy and Its Variational Principle for Locally Compact Metrizable Systems}.
Ergodic Theory and Dynamical Systems (2016), 1-26.

\bibitem{denker-et-al} M. Denker, C. Grillenberger and K. Sigmund: \
{\em Ergodic Theory on Compact Spaces}. LNM 527,
Springer-Verlag, Berlin, (1976).

\bibitem{dinaburg} E. Dinaburg: \
\emph{The relation between topological entropy and metric
  entropy}, Soviet Math. Dokl. \textbf{11} (1969), 13-16.

\bibitem{fps}  T. Ferraiol, M. Patr\~{a}o and L. Seco: \
\textit{Jordan decomposition and dynamics on flag manifolds},
Discrete Contin.\ Dyn.\ Syst.\ A, \textbf{26} (2010), 923--947.

\bibitem{goodman} T. Goodman: \
\emph{Relating topological entropy to measure entropy}, Bull.
  London. Math. Soc. \textbf{3} (1971), 176--180.

\bibitem{hilgert-neeb} J. Hilgert and K.-H. Neeb: \
{\em Structure and Geometry of Lie Groups}. SMM,
Springer-Verlag, New York, (2012).

\bibitem{koropecki}  A. Koropecki: https://mathoverflow.net/questions/248020/can-a-zero-entropy-automorphism-of-the-torus-have-a-li-yorke-pair

\bibitem{patrao:entropia} M. Patr\~ao: \
\emph{Entropy and its variational principle for non-compact metric spaces}.
Ergodic Theory and Dynamical Systems \textbf{30} (2010), 1529-1542.

\bibitem{sinai} Ya. Sinai: \
\emph{On the Notion of Entropy of a Dynamical System}.
Doklady of Russian Academy of Sciences, \textbf{124} (1959), 768--771.
\end{thebibliography}
\end{document}